\documentclass[11pt]{article}
\usepackage{enumerate}
\usepackage{amssymb,a4wide,latexsym,makeidx,epsfig,fleqn}
\usepackage{amsthm}
\usepackage{amsmath}
\usepackage{color}
\usepackage{hyperref}
\usepackage{url}
\usepackage[colorinlistoftodos]{todonotes}
\usepackage{enumerate}
\newtheorem{theorem}{Theorem}[section]

\newtheorem{definition}[theorem]{Definition}
\newtheorem{lemma}[theorem]{Lemma}

\newtheorem{corollary}[theorem]{Corollary}

\def\LEE{\mathop{\rm LEE }\nolimits}
\def\EE{\mathop{\rm EE }\nolimits}
\def\E{\mathop{\rm E }\nolimits}
\def\LE{\mathop{\rm LE }\nolimits}
\def\SLEE{\mathop{\rm SLEE }\nolimits}
\def\Tr{\mathop{\rm Tr }\nolimits}
\def\tr{\mathop{\rm tr }\nolimits}
\def\Res{\mathop{\rm Res }\nolimits}
\def\Spec{\mathop{\rm Spec }\nolimits}

\begin{document}
\textwidth 150mm \textheight 225mm
\title{Signless Laplacian Estrada index and Laplacian Estrada index of uniform hypergraphs
\thanks{Supported by the National Natural Science Foundation of China (No. 11871398) and China Scholarship Council (No. 202006290084). Email: cxduanmath@163.com; Edwin.vanDam@uvt.nl; lgwangmath@163.com}}
\author{{Cunxiang Duan$^{a,b}$, Edwin R. van Dam$^{b}$
,
Ligong Wang$^{a}$ }\\
{\small $^{a}$School of Mathematics and Statistics, Northwestern
Polytechnical University, Xi'an, P.R. China}\\
{\small $^{b}$Department of Econometrics and O.R., Tilburg University, the Netherlands}} 
\maketitle
\begin{center}
\begin{minipage}{120mm}
\vskip 0.3cm
\begin{center}
{\small {\bf Abstract}}
\end{center}
{\small
We generalize the notions of Laplacian and signless Laplacian Estrada index to uniform hypergraphs. For an $r$-uniform hypergraph $H,$ we derive an order $r+1$ trace formula of the (signless) Laplacian tensor of $H.$ Among others by using this trace formula, we obtain lower bounds for the signless Laplacian Estrada index and upper bounds for the Laplacian Estrada index. Moreover, we establish a bound involving both the Laplacian Estrada index and Laplacian energy of a uniform hypergraph.

}
\end{minipage}
\end{center}

\section{Introduction }
\label{sec:ch6-introduction}

Let $G$ be a graph with $n$ vertices and let $\lambda_{i}$ be the $i$-th eigenvalue of its adjacency matrix, for $1 \leq i \leq n.$ In 2000, Estrada \cite{E} proposed the Estrada index $\EE(G)$ of $G$, that is,
$$\EE(G)=\sum\limits_{i=1}^{n}e^{\lambda_{i}}.$$
Since the Estrada index of graphs has wide applications in biology, chemistry, physics, etc., the study of the Estrada index has attracted extensive attention.
Recall that the energy of a graph $G$ is defined as
$$\E(G)=\sum\limits_{i=1}^{n}\lvert \lambda_{i}\rvert.$$ In 2007, Pe\~{n}a, Gutman and Rada \cite{PJR} obtained upper and lower bounds on the Estrada index in terms of the number of vertices and edges, and some inequalities between $\EE(G)$ and the energy $\E(G)$ of $G.$ In 2008, Gutman \cite{G} obtained some new lower bounds on $\EE(G)$ of the graph $G$ in terms of the number of vertices and edges. In 2008, Zhou \cite{Z} also established some lower and upper bounds on $\EE(G)$ in terms of graph invariants. For more results, see \cite{CFD,D,ZZL}.

Let $G$ be a graph with $n$ vertices
and let $\mu_{i}$ be the $i$-th eigenvalue of its Laplacian matrix, for $1 \leq i \leq n.$ In 2009, Fath-Tabar, Ashrafi and Gutman \cite{FAG} defined the Laplacian Estrada index of $G$ as
$$\LEE_{1}(G)=\sum\limits_{i=1}^{n}e^{\mu_{i}}.$$
Independently, Li, Shiu and Chang \cite{LSC} defined the Laplacian Estrada index of $G$ as $$\LEE(G)=\sum\limits_{i=1}^{n}e^{\mu_{i}-\frac{2m}{n}},$$ where $m$ is the number of edges of $G$.
They also presented some bounds on $\LEE(G)$ and some relations between $\LEE(G)$ and $\LE(G),$ where $\LE(G)=\sum\limits_{i=1}^{n}\lvert \mu_{i}-\frac{2m}{n}\rvert$ is the Laplacian energy of $G.$
Note that $2m/n$ is the average of the Laplacian eigenvalues of $G$, which motivates the shift by $-2m/n$ in both $\LEE(G)$ and $\LE(G).$
However, $\LEE(G)=e^{-\frac{2m}{n}}\LEE_{1}(G),$ and therefore the two ``Laplacian Estrada indices" are in some sense equivalent. Still, the results communicated in \cite{LSC} are not equivalent to those in earlier works \cite{FAG,Z1,ZG}.

Let $G$ be a graph with $n$ vertices
and let $q_{i}$ be the $i$-th eigenvalue of its signless Laplacian matrix, for $1 \leq i \leq n.$ In 2011,  Ayyaswamy et al. \cite{ABVG} defined the signless Laplacian Estrada index of $G$ as $$\SLEE(G)=\sum\limits_{i=1}^{n}e^{q_{i}},$$
and established lower and upper bounds on $\SLEE(G)$ in terms of the number of vertices and number of edges.

For several other results on the mentioned Estrada indices, we refer to \cite{CH, DZ,DL,K}.

Because of the many results relating eigenvalues of graphs and structural properties, it is also natural to consider eigenvalues of uniform hypergraphs. For some general results about the spectrum of hypergraphs, see \cite{BFWZ,DW,DWL,Qi2,SQH}.

In 2021, Sun, Zhou and Bu \cite{SZB} generalized the definition of the Estrada index to uniform hypergraphs and obtained some bounds for it. Moreover, they studied the Estrada indices of some specific $r$-uniform hypergraphs.

In this paper, motivated by \cite{LSC,SZB}, we introduce the signless Laplacian Estrada index $\SLEE(H)$ and the Laplacian Estrada index $\LEE(H)$ of an $r$-uniform hypergraph $H.$ We also obtain an order $r+1$ trace formula of the (signless) Laplacian tensor of $H.$ Further, we establish some bounds on $\SLEE(H)$ and $\LEE(H).$ In general, the problem of finding good bounds is hindered by the fact that eigenvalues of hypergraphs may be nonreal.

This paper is organized as follows. In Section 2, we will give some basic definitions of eigenvalues of hypergraphs and preliminary results that we will use in the following sections. In Section 3, we will introduce some definitions and formulas of traces of tensors. Our main result there is a trace formula of the (signless) Laplacian tensor of an $r$-uniform hypergraph. In Section 4, we mainly obtain some lower bounds on the signless Laplacian Estrada index of uniform hypergraphs. In Section 5, we give some upper bounds on the Laplacian Estrada index of uniform hypergraphs in terms of the number of vertices, the number of edges and degree sequences. We also obtain a bound involving both the Laplacian Estrada index and Laplacian energy. Our results generalize and improve some results for graphs.

\section{Preliminaries}
\label{sec:ch-sufficient}
\subsection{Eigenvalues of hypergraphs}

A hypergraph $H$ is a pair $(V,E)$, where $E$ is a (multi-)set of subsets of $V$.
The elements of $V$ are called vertices, and the elements of $E$ are called edges. A hypergraph $H$ is said to be $r$-uniform for an integer $r\geq 2$ if each edge $e\in E$ contains precisely $r$ vertices.  An $r$-uniform hypergraph is simple if there are no repeated edges. Thus, a simple 2-uniform
hypergraph is a simple graph.
All uniform hypergraphs in this paper are simple, unless otherwise specified.

Denote by $[n]$ the set $\{1,2, \ldots,n\}.$ Let $r \geq 2$ and $r$ be even. An $r$-uniform hypergraph $H$ with $V(H)=\{v_{1}, v_{2}, \ldots, v_{n}\}$ is called odd-colorable if there exists a map $\phi : [n] \rightarrow [r]$ such that for any edge $\{v_{j_{1}}, v_{j_{2}}, \ldots , v_{j_{r}}\}$ of $H,$ we have $\phi(j_{1}) + \cdots +\phi(j_{r}) \equiv \frac{r}{2}~(\mod~ r).$ For $r=2$, odd-colorability is equivalent to bipartiteness.

In 2005, tensor eigenvalues and spectra of tensors were independently introduced by Qi \cite{Q} and Lim \cite{L}. Denote by $\mathbb{C}$ the complex field. An $r$-th order $n$-dimensional tensor $\mathcal{T}=(t_{i_{1}i_{2}\ldots i_{r}})$ is a multidimensional array (or hypermatrix), where $t_{i_{1}i_{2}\ldots i_{r}}\in \mathbb{C}$ with $ 1\leq i_{1},i_{2},\ldots,i_{r}\leq n.$ For a vector $x =(x_{1}, x_{2}, \ldots, x_{n})^{\top} \in \mathbb{C}^{n}$, $\mathcal{T} x^{r-1}$ is defined as a vector in $\mathbb{C}^{n}$ with $i$-th component being
$$(\mathcal{T}x^{r-1})_{i}=\sum\limits^{n}_{i_{2},\ldots,i_{r}=1}t_{ii_{2}\ldots i_{r}}x_{i_{2}}\ldots x_{i_{r}},~~ \mbox{for}~ i\in[n].$$

Let $x^{[r-1]}=(x_{1}^{r-1}, x_{2}^{r-1}, \ldots, x_{n}^{r-1})^{\top}\in \mathbb{C}^{n}$. A number $\lambda\in \mathbb{C}$ is called an eigenvalue of the tensor $\mathcal{T}$ if there exists a nonzero vector $x\in \mathbb{C}^{n}$ such that $$\mathcal{T} x^{r-1}=\lambda x^{[r-1]};$$ in this case, $x$ is called an eigenvector of $\mathcal{T}$ corresponding to the eigenvalue $\lambda.$ The spectral radius of the tensor $\mathcal{T}$ is defined as $\rho(\mathcal{T})=\max\{|\lambda|~|~\lambda$ is an eigenvalue of $\mathcal{T}\}$. We note that the spectral radius of a nonnegative tensor is an eigenvalue \cite{Qi2}. We use $\Spec(\mathcal{T})$ to denote the set of all eigenvalues of $\mathcal{T}.$

For an $r$-th order $n$-dimensional tensor $\mathcal{T},$ the characteristic polynomial $\phi_{\mathcal{T}}(\lambda)$ of the tensor $\mathcal{T}$ is defined as the resultant $\Res(\lambda x^{[r-1]}-\mathcal{A} x^{r-1}).$

In this paper, from now on, we only consider real symmetric tensors.

\noindent\begin{definition}\label{de:c1} \cite{CoDu}
Let $H = (V, E)$ be an $r$-uniform hypergraph with $n$ vertices. The adjacency tensor of $H$ is defined as an $r$-th order $n$-dimensional tensor $\mathcal{A}_{H}=(a_{i_{1}i_{2}\ldots i_{r}}),$ where $$a_{i_{1}i_{2}\ldots i_{r}}=\left\{
\begin{array}{ll}
\frac{1}{(r-1)!},& \mbox {if}   ~\{i_{1},i_{2},\ldots, i_{r}\} \in E,
\\
0,& \mbox {otherwise}.
\end{array}
\right.$$
\end{definition}

The degree of a vertex $v_{i}$ of $H$, denoted by $d_{i},$ is the number of edges containing $v_{i}.$ Let $\mathcal{D}_{H}$ be an $r$-th order $n$-dimensional diagonal tensor with its diagonal element $d_{ii\ldots i}$ being $d_{i}$, the degree of $v_{i}$, for $i\in [n].$ Denote by $\mathcal{L}_{H}=\mathcal{D}_{H}-\mathcal{A}_{H}$ and $\mathcal{Q}_{H}=\mathcal{D}_{H}+\mathcal{A}_{H}$ the Laplacian tensor and the signless Laplacian tensor of the hypergraph $H$, respectively. The eigenvalues and (signless) Laplacian eigenvalues of $H$ refer to the eigenvalues of its adjacency tensor and its (signless) Laplacian tensor. We let $s=n(r-1)^{n-1}$ and use $q_{i}$ and $\mu_{i}$, for $i\in [s]$, to denote the signless Laplacian eigenvalues and the Laplacian eigenvalues of $H,$ respectively. We note that eigenvalues of symmetric tensors such as $\mathcal{A}_{H}, \mathcal{L}_{H},$ and $\mathcal{Q}_{H}$, may be nonreal, contrary to eigenvalues of graphs (the case $r=2$). If an eigenvalue is indeed nonreal, then also its conjugate is an eigenvalue. Moreover, the (signless) Laplacian eigenvalues lie in the disk with center and radius $\Delta$, the maximum degree of $H$ \cite{Qi2}, and thus they have a nonnegative real part.

\subsection{Some basic lemmas}

In this section, we provide some elementary lemmas, starting with some inequalities for sequences.
For an $r$-uniform hypergraph with degree sequence $d_{1}, d_{2}, \ldots, d_{n},$ we use $M$ to denote the first Zagreb index, that is, $M=\sum\limits_{i=1}^{n}d_{i}^{2}.$ The number of edges $m$ satisfies $m=\frac{1}{r}\sum\limits_{i=1}^{n}d_{i}$.

\noindent\begin{lemma}\label{le:2-1}
Let $H$ be an $r$-uniform hypergraph with degree sequence $d_{1}\geq d_{2}\geq \cdots\geq d_{n}$ and $m$ edges. Let $n\geq 3.$ Then
\begin{align*}
\tfrac{r^{2}m^{2}}{n} \leq M\leq r^{2}m^{2}-n(n-1)d_{n}^{2},
\end{align*}
with equalities holding if and only if $H$ is regular.
\end{lemma}

\begin{proof}
By Cauchy's inequality, we have $$M=\sum_{i=1}^{n}d_{i}^{2}\geq \tfrac{1}{n}(\sum\limits_{i=1}^{n}d_{i})^{2}=\tfrac{r^{2}m^{2}}{n},$$ with equality if and only if $d_{i}$ is constant.

For the second inequality, note that
$$M=\sum_{i=1}^{n}d_{i}^{2}=(\sum_{i=1}^{n}d_{i})^{2}-2\sum\limits_{i<j}d_{i}d_{j}\leq r^{2}m^{2}-n(n-1)d_{n}^{2},$$
with equality if and only if $d_{i}=d_j=d_n$ for all $i<j$.
\end{proof}

\noindent\begin{lemma}\label{le:2-2}
Let $H$ be an $r$-uniform hypergraph with degree sequence $d_{1}\geq d_{2}\geq \cdots\geq d_{n}$ and $m$ edges. Then
\begin{align*}
M\leq (d_{1}+d_{n})rm-nd_{1}d_{n},
\end{align*}
with equality if and only if $d_{i}=d_{1}$ or $d_{i}=d_{n}$ for $i=2,3, \ldots,n-1.$
\end{lemma}

\begin{proof}  From the inequality $$\sum_{i=1}^{n}(d_{i}-d_{1})(d_{i}-d_{n})\leq 0$$
it follows that $$\sum_{i=1}^{n}d_{i}^{2}\leq (d_{1}+d_{n})\sum_{i=1}^{n}d_{i}-nd_{1}d_{n}=(d_{1}+d_{n})rm-nd_{1}d_{n},$$
with equality if and only if $d_{i}=d_{1}$ or $d_{i}=d_{n}$ for $i=2,3, \ldots,n-1.$
\end{proof}

\noindent\begin{lemma}\label{le:2-3}
 Let $H$ be an $r$-uniform hypergraph with degree sequence $d_{1}, d_{2}, \ldots,d_{n},$ and let $k\geq 2$ be a positive integer. Then $$n(\tfrac{1}{n}\sum_{i=1}^{n}d_{i})^{k}\leq \sum_{i=1}^{n}d_{i}^{k}\leq 
 M^{\frac{k}{2}}.$$
\end{lemma}

\begin{proof}  The first inequality follows from H\"{o}lder's inequality. Moreover, it is clear that $$\sum_{i=1}^{n}d_{i}^{k}=\sum_{i=1}^{n}(d_{i}^{2})^{\frac{k}{2}}\leq (\sum_{i=1}^{n}d_{i}^{2})^{\frac{k}{2}}=M^{\frac{k}{2}}.$$
\end{proof}

\noindent\begin{lemma}\label{le:2-5}\cite{YQSO}
 Let $H$ be a non-empty $r$-uniform hypergraph. Then $Spec(\mathcal{L}_{H}) = Spec(\mathcal{Q}_{H})$ if and only if $r$ is even and $H$ is odd-colorable.
\end{lemma}

\noindent\begin{lemma}\label{le:2-6}
 Let $f_{r}(x)=e^{x}-\sum\limits_{k=0}^{r}\tfrac{1}{k!}x^{k}.$ If $r$ is odd, then $f_{r}(x)\geq f_{r}(0)=0$ for all $x\in \mathbb{R}.$
\end{lemma}
Note that for $r$ even, $f_{r}(x)\geq f_{r}(0)=0$ for all $x\geq 0,$ but $f_{r}(x)<f_{r}(0)$ for $x<0.$ This causes a mistake in the proof of \cite[lower bound Thm.~5]{LSC}, hence this result may not be correct.

\section{Traces of uniform hypergraphs}
\label{sec:ch-inco}

Traces of tensors are important invariants in the spectral theory of tensors. Morozov and Shakirov \cite{MS} defined the $j$-th order trace $\Tr_{j}(\mathcal{T})$ of an $r$-th order $n$-dimensional tensor $\mathcal{T}=(t_{i_{1}i_{2}\cdots i_{r}})$ as $$\Tr_{j}(\mathcal{T})=(r-1)^{n-1}\sum_{j_{1}+j_{2}+\cdots+j_{n}=j}\big[\prod_{i=1}^{n}\tfrac{1}{(j_{i}(r-1))!}\big(\sum_{l \in [n]^{r-1}}t_{il}\tfrac{\partial}{\partial x_{il}}\big)^{j_{i}}\big] \tr (X^{j(r-1)}),$$
where $X = (x_{ij})$ is an $n\times n$ auxiliary matrix, $\frac{\partial}{\partial x_{il}} = \frac{\partial}{\partial x_{il_{2}}}\frac{\partial}{\partial x_{il_{3}}}\cdots \frac{\partial}{\partial x_{il_{r}}}$ if $l = l_{2}\cdots l_{r},$ and  $j_{1}, \ldots, j_{n}$ run over all nonnegative integers with $j_{1} +\cdots+j_{n} =j.$ For any $j\in [s],$ Hu et al. \cite{HHLQ} showed that
$$\Tr_{j}(\mathcal{T})=\sum\limits_{\lambda \in \Spec(\mathcal{T})}\lambda^{j},$$
where $\Spec(\mathcal{T})$ is the spectrum of the tensor $\mathcal{T}.$

For an $r$-uniform hypergraph $H,$ Cooper and Dutle \cite{CoDu} proved that $\Tr_{j}(\mathcal{A}_{H})=0$ for $j=1,2,\ldots,r-1.$ An $r$-uniform hypergraph $H$ is called $r$-valent if the degree of every vertex of $H$ is the multiple of $r.$ A simplex in a hypergraph is a set of $r+1$ vertices such that every subset of $r$ vertices forms an edge. The simplex is the only $r$-uniform $r$-valent multihypergraph with $r+1$ edges. Cooper and Dutle \cite{CoDu} showed that $\Tr_{r+1}(\mathcal{A}_{H})=C_{r}(r+1)(r-1)^{n-r}$(\# of simplices in $H$), where $C_{r}$ is a constant depending only on $r.$ They also showed that $C_{3}=21,~C_{4}=588,~C_{5}=28230.$ For graphs, it is well known \cite{Bi} that $\Tr_{3}(\mathcal{A}_{G})=6$(\# of triangles in $G$), i.e., $C_{2}=2.$ For given (relatively small) $r$, the constant $C_{r}$ can be explicitly determined by some recent results \cite{CC}. For example, $C_{6}=2092206,~C_{7}= 220611384,~C_{8}= 31373370936.$

Zhou et al. \cite{ZSWB} obtained some trace formulas for the signless Laplacian tensor and the Laplacian tensor of uniform hypergraphs as follows.
\noindent\begin{lemma}\label{le:3-1}\cite{ZSWB}
Let $H$ be an $r$-uniform hypergraph with degree sequence $d_{1}, d_{2}, \ldots,d_{n}$ and $m$ edges. Then
$$\Tr_{j}(\mathcal{L}_{H})=\Tr_{j}(\mathcal{Q}_{H})=(r-1)^{n-1}\sum_{i=1}^{n}d_{i}^{j},~~j=1,2,\ldots,r-1,$$
$$\Tr_{r}(\mathcal{L}_{H})=(-1)^{r}r^{r-1}(r-1)^{n-r}m+(r-1)^{n-1}\sum_{i=1}^{n}d_{i}^{r},$$ and
$$\Tr_{r}(\mathcal{Q}_{H})=r^{r-1}(r-1)^{n-r}m+(r-1)^{n-1}\sum_{i=1}^{n}d_{i}^{r}.$$
\end{lemma}

Let $j$ be a positive integer and $\mathcal{F}_{j}=\{(i_{1}\alpha_{1},\ldots,i_{j}\alpha_{j}) \mid 1 \leq i_{1} \leq i_{2} \leq\cdots \leq i_{j}\leq n;~\alpha_{1}, \ldots, \alpha_{j}\in [n]^{r-1}\}.$ For $F=(i_{1}\alpha_{1},\ldots,i_{j}\alpha_{j})\in \mathcal{F}_{j}$ and a tensor $\mathcal{T}=(t_{i_{1}i_{2}\cdots i_{r}}),$ let $\pi_{F}(\mathcal{T})=t_{i_{1}\alpha_{1}}\cdots t_{i_{j}\alpha_{j}}.$
\noindent\begin{definition}\label{de:3-1} \cite{SQH}
Let $F=(i_{1}\alpha_{1},\ldots,i_{j}\alpha_{j})\in \mathcal{F}_{j},$ where $i_{h}\alpha_{h}\in [n]^{r}$, $h=1,2, \ldots, j.$ Then
\\$(1)$ Let $E(F)=\bigcup\limits_{h=1}^{j}E_{h}(F),$ where $E_{h}(F)$ is the arc multi-set $E_{h}(F)=\{ (i_{h},v_{1}), \ldots, (i_{h},v_{r-1}) \}$ if $\alpha_{h}=v_{1}\cdots v_{r-1}.$
\\$(2)$ Let $b(F)$ be the product of the factorials of the multiplicities of all the arcs of $E(F)$.
\\$(3)$ Let $c(F)$ be the product of the factorials of the outdegrees of all the vertices in the arc multi-set $E(F)$.
\\$(4)$ Let $W(F)$ be the set of all Eulerian closed walks $W$ with the arc multi-set $E(F).$
\end{definition}
In most of the literature, the term closed walk is used instead of Eulerian closed walk. We prefer the term Eulerian closed walk to emphasize that all the arcs should be used in such a closed walk.

\noindent\begin{lemma}\label{le:3-2}\cite{SQH}
Let $\mathcal{T}=(t_{i_{1}i_{2}\cdots i_{r}})$ be an $r$-th order $n$-dimensional tensor. Then
$$\Tr_{j}(\mathcal{T})=(r-1)^{n-1}\sum_{F\in \mathcal{F}'_{j}}\tfrac{b(F)}{c(F)}\pi_{F}(\mathcal{T})\lvert W(F)\rvert,$$
where $\mathcal{F}'_{j}=\{F\in  \mathcal{F}_{j} \mid F$ is $r$-valent$\}.$
\end{lemma}

Note that multiple arcs (or loops) are not distinguished if we determine $\lvert W(F)\rvert.$ Note also that if $F$ is not $r$-valent, then $W(F)=\emptyset$.
We now present a new result for the order $r+1$ trace of the (signless) Laplacian.

\noindent\begin{theorem}\label{th:3-1}
Let $H$ be an $r$-uniform hypergraph with degree sequence $d_{1}, d_{2}, \ldots, d_{n}$. Then
\begin{align*}
\Tr_{r+1}(\mathcal{L}_{H})&=(-1)^{r+1}\Tr_{r+1}(\mathcal{A}_{H})+(r-1)^{n-1}\sum_{i=1}^{n}d_{i}^{r+1}+(-1)^{r}(r-1)^{n-r}(r+1)r^{r-2}\sum_{i=1}^{n}d_{i}^{2},
\end{align*}
and
\begin{align*}
\Tr_{r+1}(\mathcal{Q}_{H})=\Tr_{r+1}(\mathcal{A}_{H})+(r-1)^{n-1}\sum_{i=1}^{n}d_{i}^{r+1}+(r-1)^{n-r}(r+1)r^{r-2}\sum_{i=1}^{n}d_{i}^{2}.
\end{align*}
\end{theorem}
Recall that $\Tr_{r+1}(\mathcal{A}_{H})=C_{r}(r+1)(r-1)^{n-r}$(\# of simplices in $H$), and $C_{r}$ is a constant depending only on $r.$

\begin{proof}  By Lemma \ref{le:3-2}, we have $$\Tr_{r+1}(\mathcal{L}_{H})=(r-1)^{n-1}\sum_{F\in \mathcal{F}'_{r+1}}\tfrac{b(F)}{c(F)}\pi_{F}(\mathcal{L}_{H})\lvert W(F)\rvert,$$
where $\mathcal{F}'_{r+1}=\{F\in  \mathcal{F}_{r+1} \mid F$ is $r$-valent$\}.$

First, we note that if $F\in \mathcal{F}'_{r+1}$ does not contain diagonal elements, then $F$ corresponds to a simplex, and these $F$ are precisely the ones that contribute to $\Tr_{r+1}(\mathcal{A}_{H}).$ The (only) difference is that $\pi_{F}(\mathcal{L}_{H})=(-1)^{r+1}\pi_{F}(\mathcal{A}_{H}),$ hence the total contribution (over all such $F$) to $\Tr_{r+1}(\mathcal{L}_{H})$ is $(-1)^{r+1}\Tr_{r+1}(\mathcal{A}_{H})$. What remains to consider are those $F\in \mathcal{F}'_{r+1}$ that do contain a diagonal element, and for which  $\pi_{F}(\mathcal{L}_{H})\neq 0$ and $\lvert W(F)\rvert\neq 0.$ We now note that if $\lvert W(F)\rvert\neq 0$, then $E(F)$ induces a connected multigraph. If $F$ contains a diagonal element, $ii \ldots i$ say, then there are two cases.

In the first case, vertex $i$ is not contained in any edge of $F$. Then $F=(ii\cdots i, \ldots, ii\cdots i)$ because the multigraph corresponding to $E(F)$ must be connected. For all such $F,$ we have that $b(F)=c(F)=((r+1)(r-1))!, ~\pi_{F}(\mathcal{L}_{H})=d_{i}^{r+1},$ and $\lvert W(F)\rvert=1,$ hence the total contribution over all vertices $i$ in such $F$ to $\Tr_{r+1}(\mathcal{L}_{H})$ equals $$(r-1)^{n-1}\sum_{i=1}^{n}d_{i}^{r+1}.$$

In the second case, vertex $i$ is contained in an edge of $F$, $\{i,i_{2},\ldots ,i_{r}\}$ say. Because $F$ is $r$-valent, all vertices $i,i_{2},\ldots ,i_{r}$ must appear in all of the remaining $r-1$ egdes of $F$. Moreover, each vertex must have outdegree at least one in $E(F)$ (otherwise there are no closed walks). So $F$ is an appropriate ordering of $\{ii\cdots i, i\alpha_{1}, i_{2}\alpha_{2}, \ldots, i_{r}\alpha_{r}\}$, where each of the $i_{h}\alpha_{h}$ represents the same edge as $i\alpha_{1}$, $2 \leq h \leq r.$ For each edge $e$ of $H$ containing $i,$ there are two orderings of the first entry, and $[(r-1)!]^{r}$ orderings of the $\alpha_{h}$ (in total). Thus, the total number of $F$ (for fixed $i$) of the given form is $2d_{i}[(r-1)!]^{r}.$ We will next show that each such $F$ has the same contribution to $\Tr_{r+1}(\mathcal{L}_{H}).$ Indeed, (only depending on $i$) first of all it follows that $E(F)$ induces a complete directed graph on $r$ vertices with $r-1$ loops added at one vertex $i.$ Thus, $b(F)=(r-1)!$ and $c(F)=(2r-2)![(r-1)!]^{r-1}.$ Moreover, $\pi_{F}(\mathcal{L}_{H})=(-1)^{r}\frac{d_{i}}{[(r-1)!]^{r}},$ so what remains is to count the number of Eulerian closed walks on $E(F)$. In order to do so, we apply the so-called BEST theorem \cite{AB} and the fact that the number of spanning trees on a complete digraph on $r$ vertices equals $r^{r-2}$. It follows that if we label the loops, then the number of Eulerian cycles on the digraph induced by $E(F)$ equals $(2r-3)![(r-2)!]^{r-1}r^{r-2}$. Because each Eulerian cycle has $r^{2}-1$ arcs (including the $r-1$ loops) and loops in $W(F)$ are not labelled, it follows that $\lvert W(F)\rvert=\frac{(r^{2}-1)(2r-3)![(r-2)!]^{r-1}r^{r-2}}{(r-1)!}.$ Thus, the total contribution of all these $F$ to $\Tr_{r+1}(\mathcal{L}_{H})$ equals
\begin{align*}
&(r-1)^{n-1}\tfrac{(r-1)!}{(2r-2)![(r-1)!]^{r-1}}\tfrac{(r^{2}-1)(2r-3)![(r-2)!]^{r-1}r^{r-2}}{(r-1)!}
\sum_{i=1}^{n}(-1)^{r}\tfrac{d_{i}}{[(r-1)!]^{r}}2d_{i}[(r-1)!]^{r}
\\&=(-1)^{r}(r-1)^{n-r}(r+1)r^{r-2}\sum_{i=1}^{n}d_{i}^{2}.
\end{align*}
By collecting all contributions, the claimed result follows.

For the signless Laplacian tensor, the only difference is in the sign of $\pi_{F}(\mathcal{Q}_{H}),$ and it follows that $\pi_{F}(\mathcal{L}_{H})=(-1)^{r+1}\pi_{F}(\mathcal{Q}_{H})$ if $F$ has no loops, and $\pi_{F}(\mathcal{L}_{H})=(-1)^{r}\pi_{F}(\mathcal{Q}_{H})$ in the above second case. The result for the signless Laplacian tensor thus follows.
\end{proof}

\section{Lower bounds on the signless Laplacian Estrada index of uniform hypergraphs}
\label{sec:ch-inco}
In this section, we define the signless Laplacian Estrada index, which generalizes the signless Laplacian Estrada index of graphs \cite{ABVG}. By using trace formulas, we then obtain a lower bound on the signless Laplacian Estrada index of a uniform hypergraph. Further, for $r$-uniform odd-colorable hypergraphs, we present a lower bound on the signless Laplacian Estrada index in terms of the number of vertices and the degree sequence.

Motivated by the proofs of the above results, we also obtain a lower bound on the Laplacian Estrada index of a (usual) graph $G$, which improves a known bound on graphs in some cases.

Let $H$ be an $r$-uniform hypergraph with $n$ vertices and $m$ edges. The signless Laplacian Estrada index is defined as $$\SLEE(H)=\sum_{i=1}^{s}e^{q_{i}}.$$

\noindent\begin{theorem}\label{th:4-1}
Let $H$ be an $r$-uniform hypergraph with $n$ vertices, $m$ edges, and degrees $d_{1}, d_{2} ,\ldots, d_{n}$. Then
\begin{align*}
\SLEE(H) \geq (r-1)^{n-1}\Big(n&+(\tfrac{r}{r-1})^{r-1}\tfrac{m}{r!}+\tfrac{1}{(r-1)^{n-1}(r+1)!}\Tr_{r+1}(\mathcal{A}_{H})
\\&+\tfrac{(r+1)r^{r-2}}{(r-1)^{r-1}(r+1)!}\sum_{i=1}^{n}d_{i}^{2}+\sum_{k=1}^{r+1}\tfrac{1}{k!}\sum_{i=1}^{n}d_{i}^{k}\Big).
\end{align*}
Equality is attained if and only if $H$ is an empty hypergraph.
\end{theorem}

\begin{proof}  By means of a power-series expansion, we have
\begin{align*}
\SLEE(H)&=\sum_{i=1}^{s}\sum_{k\geq 0}\tfrac{1}{k!}q_{i}^{k}=\sum_{k\geq 0}\tfrac{1}{k!}\sum\limits_{i=1}^{s}q_{i}^{k}=\sum_{k\geq 0}\tfrac{1}{k!}\Tr_{k}(\mathcal{Q}_{H}).
\end{align*}

By the trace formula of tensors (Lemma \ref{le:3-2}), we know that $\Tr_{k}(\mathcal{Q}_{H})$ is a nonnegative real number, hence

$$\SLEE(H) \geq \sum_{k= 0}^{r+1}\tfrac{1}{k!}\Tr_{k}(\mathcal{Q}_{H}).$$

By Lemma \ref{le:3-1} and Theorem \ref{th:3-1}, the claimed inequality now follows.
Equality holds if and only if $\Tr_{k}(\mathcal{Q}_{H})=0$ for $k\geq r+2$, which is the case if and only if all eigenvalues are zero. 
Thus, by $\Tr_{1}(\mathcal{Q}_{H})=(r-1)^{n-1}rm=0,$ we know $m=0$, i.e., $H$ is an empty hypergraph.
\end{proof}

For odd-colorable $r$-uniform hypergraphs, we know that $\Spec(\mathcal{L}_{H})=\Spec(\mathcal{Q}_{H})$ and hence, by Theorem \ref{th:3-1}, that $\Tr_{r+1}(\mathcal{A}_{H})=0.$ The latter also follows directly from the fact that $\mathcal{A}_{H}$ has spectrum that is symmetric about the origin \cite{N}. Thus we have the following.

\noindent\begin{corollary}\label{co:4-2}
Let $r$ be even and let $H$ be an $r$-uniform odd-colorable hypergraph with degrees $d_{1}, d_{2} ,\ldots, d_{n}$ and $m$ edges. Then
\begin{align*}
\SLEE(H) \geq (r-1)^{n-1}&\Big(n+(\tfrac{r}{r-1})^{r-1}\tfrac{m}{r!}
+\tfrac{(r+1)r^{r-2}}{(r-1)^{r-1}(r+1)!}\sum\limits_{i=1}^{n}d_{i}^{2}+\sum\limits_{k=1}^{r+1}\tfrac{1}{k!}\sum\limits_{i=1}^{n}d_{i}^{k}\Big).
\end{align*}
Equality is attained if and only if $H$ is an empty hypergraph.
\end{corollary}

The generalization of the Laplacian Estrada index $\LEE_{1}(G)$ (as introduced by \cite{FAG} for graphs) to hypergraphs could be defined as $$\LEE_{1}(H)= \sum_{i=1}^{s}e^{\mu_{i}}.$$ By Lemma \ref{le:2-5}, note that $\SLEE(H)=\LEE_{1}(H)$ for an $r$-uniform odd-colorable hypergraph $H$. Therefore, the inequality in Corollary \ref{co:4-2} also holds for the Laplacian Estrada index $\LEE_{1}(H)$ for odd-colorable hypergraphs.

We next turn to (usual) graphs and consider the (other) Laplacian Estrada index $\LEE(G)$. Later, we will also generalize this index to hypergraphs, but we first wish to apply a similar technique as the above to graphs, by applying Lemma \ref{le:2-6}. As this relies on the eigenvalues being real, we cannot fully generalize this to hypergraphs.

\noindent\begin{theorem}\label{th:4-3}
Let $G$ be a graph with degrees $d_{1}, d_{2} ,\ldots, d_{n}$, $m$ edges, and $\eta$ triangles. Then
$$\LEE(G) \geq n+\tfrac{1}{6}\sum_{i=1}^{n}d_{i}^{3}+(1-\tfrac{m}{n})\sum_{i=1}^{n}d_{i}^{2}+m+\tfrac{8m^{3}}{3n^{2}}-\tfrac{4m^{2}}{n}-\eta.$$
\end{theorem}

\begin{proof}  Since $\sum\limits_{i=1}^{n}\mu_{i}=2m$ and $\Tr_3(\mathcal{A}_G)=6\eta$, by Lemmas \ref{le:2-6}, \ref{le:3-1}, and Theorem \ref{th:3-1}, we have that
\begin{align*}
\LEE(G)&=\sum_{i=1}^{n}e^{\mu_{i}-\frac{2m}{n}}
\geq\sum_{k=0}^{3}\tfrac{1}{k!}\sum\limits_{i=1}^{n}(\mu_{i}-\tfrac{2m}{n})^{k}\\
&=n+\tfrac{1}{2}\sum\limits_{i=1}^{n}(\mu_{i}-\tfrac{2m}{n})^{2}+\tfrac{1}{6}\sum\limits_{i=1}^{n}(\mu_{i}-\tfrac{2m}{n})^{3}\\
&=n+\tfrac{2m^2}{n}-\tfrac{4m^3}{3n^2}+(\tfrac{2m^2}{n^2}-\tfrac{2m}{n})\Tr_1(\mathcal{L}_G)+(\tfrac{1}{2}-\tfrac{m}{n})\Tr_2(\mathcal{L}_G)+\tfrac{1}{6}\Tr_3(\mathcal{L}_G)\\
&=n+\tfrac{1}{6}\sum_{i=1}^{n}d_{i}^{3}+(1-\tfrac{m}{n})\sum_{i=1}^{n}d_{i}^{2}+m+\tfrac{8m^{3}}{3n^{2}}-\tfrac{4m^{2}}{n}-\eta.
\end{align*}
\end{proof}

\noindent{\bf Remark.} Only when $n \geq m$, we seem to be able to derive the claimed --- yet unproven\footnote{see the remark after Lemma \ref{le:2-6}} --- bound $\sqrt{n^2+4m}$ of \cite{LSC}. Indeed, if we assume for simplicity that $G$ is connected and $n\geq m\geq 3$, then $\eta\leq 1.$ By Lemmas \ref{le:2-1}, \ref{le:2-3} and Theorem \ref{th:4-3}, we then have that
\begin{align*}
\LEE(G)& \geq n+\tfrac{1}{6}\sum_{i=1}^{n}d_{i}^{3}+(1-\tfrac{m}{n})\sum_{i=1}^{n}d_{i}^{2}+m+\tfrac{8m^{3}}{3n^{2}}-\tfrac{4m^{2}}{n}-1\cr
&\geq n+\tfrac{4m^{3}}{3n^{2}}+(1-\tfrac{m}{n})\tfrac{4m^{2}}{n}+m+\tfrac{8m^{3}}{3n^{2}}-\tfrac{4m^{2}}{n}-1\cr
&=n+m-1> \sqrt{n^{2}+4m}.
\end{align*}

As mentioned before, the Laplacian eigenvalues of uniform hypergraphs could be nonreal. Since Lemma \ref{le:2-6} and a lot of similar inequalities hold only for real numbers, we may need a new technique to obtain a good lower bound on the Laplacian Estrada index of hypergraphs.

\section{Upper bounds on the Laplacian Estrada index of uniform hypergraphs}
\label{sec:ch-inco}
In this section,  we define the Laplacian Estrada index of uniform hypergraphs, which generalizes the Laplacian Estrada index of graphs from \cite{LSC}. We obtain some upper bounds on the Laplacian Estrada index $\LEE(H)$ and some inequalities between $\LEE(H)$ and the Laplacian energy $\LE(H)$ of an $r$-uniform hypergraph $H.$ These generalize and improve some known results for graphs.

Let $H$ be an $r$-uniform hypergraph with $n$ vertices and $m$ edges. The Laplacian Estrada index of $H$ is defined as  $$\LEE(H)=\sum_{i=1}^{s}e^{\mu_{i}-\frac{rm}{n}}.$$
By means of a power-series expansion, we have $$\LEE(H)=\sum_{i=1}^{s}\sum_{k\geq 0}\tfrac{1}{k!}(\mu_{i}-\tfrac{rm}{n})^{k}=\sum_{k\geq 0}\tfrac{1}{k!}\sum_{i=1}^{s}(\mu_{i}-\tfrac{rm}{n})^{k}=\sum_{k\geq 0}\tfrac{1}{k!}M_{k},$$
where $M_{k}=\sum\limits_{i=1}^{s}(\mu_{i}-\tfrac{rm}{n})^{k}.$

Note that $M_k=\Tr_k(\mathcal{L}_H-\tfrac{rm}{n}\mathcal{I})$, where $\mathcal{I}$ is the identity tensor, so by the trace formula of Lemma \ref{le:3-2}, we know that $M_k$ is a real number. From Lemma \ref{le:3-1}, it follows that $M_1=\sum\limits_{i=1}^{s}(\mu_{i}-\tfrac{rm}{n})=0$.

\noindent\begin{theorem}\label{th:5-1}
Let $H$ be an $r$-uniform hypergraph with $n$ vertices and $m$ edges. Then
$$\LEE(H) \leq s-1-\sqrt{\Tr_{2}(\mathcal{L}_{H})-(r-1)^{n-1}\tfrac{r^{2}m^{2}}{n}}+e^{\sqrt{s\rho(\mathcal{L}_{H})^{2}-(r-1)^{n-1}\frac{r^{2}m^{2}}{n}}},$$
where $\rho(\mathcal{L}_{H})$ is the Laplacian spectral radius of $H$. Equality is attained if and only if $H$ is an empty hypergraph.
\end{theorem}

\begin{proof} Let $\mu_{i}=a_{i}+b_{i}\mathbf{i}$ be the $i$-th eigenvalue of $\mathcal{L}_{H},$ with $a_{i}, b_{i}\in \mathbb{R}$ and $\mathbf{i}^{2}=-1.$ By Lemma \ref{le:3-1}, we have
$$\Tr_1(\mathcal{L}_{H})=\sum_{i=1}^{s}\mu_{i}=\sum_{i=1}^{s}a_{i}=(r-1)^{n-1}rm.$$
Moreover, 
\begin{align}\label{eq:5-1}
\Tr_2(\mathcal{L}_{H})=\sum_{i=1}^{s}\mu_{i}^{2}=\sum_{i=1}^{s}(a_{i}^{2}-b_{i}^{2}) \leq \sum_{i=1}^{s}a_{i}^{2},
\end{align}
with equality holding for $r=2$.

Next, we define $A_2=\sum_{i=1}^{s}\lvert \mu_{i}-\tfrac{rm}{n}\rvert^{2}$. We then have that
\begin{align}\label{eq:5-2}
A_2=\sum_{i=1}^{s}\left((a_{i}-\tfrac{rm}{n})^{2} +b_{i}^{2}\right)
= \sum_{i=1}^{s}(a_{i}^{2}+b_{i}^{2})-(r-1)^{n-1}\tfrac{r^{2}m^{2}}{n}
\leq s\rho(\mathcal{L}_{H})^{2}-(r-1)^{n-1}\tfrac{r^{2}m^{2}}{n}.
\end{align}

Using among others \eqref{eq:5-1} and \eqref{eq:5-2}, it now follows that
\begin{align*}
\LEE(H)&=s+\sum\limits_{k\geq 2}\tfrac{1}{k!}\sum_{i=1}^{s}(\mu_{i}-\tfrac{rm}{n})^{k}\leq s+\sum\limits_{k\geq 2}\tfrac{1}{k!}\sum_{i=1}^{s}\lvert \mu_{i}-\tfrac{rm}{n}\rvert^{k}\\
&= s+\sum\limits_{k\geq 2}\tfrac{1}{k!}\sum_{i=1}^{s}\left(\lvert \mu_{i}-\tfrac{rm}{n}\rvert^{2}\right)^{\tfrac{k}{2}}\leq s+\sum\limits_{k\geq 2}\tfrac{1}{k!}(A_2)^{\frac{k}{2}}\\
&\leq s-1-\sqrt{\sum_{i=1}^{s}(a_{i}^{2}+b_{i}^{2})-(r-1)^{n-1}\tfrac{r^{2}m^{2}}{n}}
+e^{\sqrt{s\rho(\mathcal{L}_{H})^{2}-(r-1)^{n-1}\frac{r^{2}m^{2}}{n}}}\\&
\leq s-1-\sqrt{\Tr_{2}(\mathcal{L}_{H})-(r-1)^{n-1}\tfrac{r^{2}m^{2}}{n}}+e^{\sqrt{s\rho(\mathcal{L}_{H})^{2}-(r-1)^{n-1}\frac{r^{2}m^{2}}{n}}}.
\end{align*}

If equality is attained, then it follows from the first inequality that $\mu_{i}-\tfrac{rm}{n}$ is a nonnegative real number for all $i$. Because $0$ is a Laplacian eigenvalue of $H,$ it follows that $m=0$, and hence $H$ is an empty hypergraph.
\end{proof}

Recall that for $r \geq 3$, we have that $\Tr_{2}(\mathcal{L}_{H})=(r-1)^{n-1}M$.

For $r=2$ and a graph $G$, $\Tr_{2}(\mathcal{L}_{G})=2m+M$. Moreover, the eigenvalues of $\mathcal{L}_{G}$ are nonnegative real numbers and $\sum\limits_{i=1}^{n}a_{i}^{2} =\Tr_{2}(\mathcal{L}_{G}).$ Combining this with the proof of Theorem \ref{th:5-1}, we have $$\LEE(G) \leq n-1-\sqrt{\Tr_{2}(\mathcal{L}_{G})-\tfrac{4m^{2}}{n}}+e^{\sqrt{\Tr_{2}(\mathcal{L}_{G})-\frac{4m^{2}}{n}}}.$$ By Lemma \ref{le:2-2}, we know this improves the upper bound of \cite[Thm.~5]{LSC}.

By Lemmas \ref{le:2-1}, \ref{le:2-2}, and Theorem \ref{th:5-1}, it is straightforward to obtain the following theorem.
\noindent\begin{theorem}\label{th:5-2}
Let $r \geq 3$. Let $H$ be an $r$-uniform hypergraph with $n$ vertices, $m$ edges, maximal degree $d_{1}$, and minimal degree $d_{n}$. Then
$$\LEE(H) \leq s-1-\sqrt{(r-1)^{n-1}(\theta-\tfrac{r^{2}m^{2}}{n})}+e^{\sqrt{s\rho(\mathcal{L}_{H})^{2}-(r-1)^{n-1}\frac{r^{2}m^{2}}{n}}},$$
where $\theta=\min\{rm(d_{1}+d_{n})-nd_{1}d_{n}, r^{2}m^{2}-n(n-1)d_{n}^{2}\}$ and $\rho(\mathcal{L}_{H})$ is the Laplacian spectral radius of $H$. Equality is attained if and only if $H$ is an empty hypergraph.
\end{theorem}

We note that Theorem \ref{th:5-1} essentially improves also the bound of \cite[Eq.~17]{LSC}, which involves also the Laplacian energy of a graph.

Recall that the Laplacian energy of an $r$-uniform hypergraph $H$ with $n$ vertices and $m$ edges is defined as $$\LE(H)=\sum_{i=1}^{s}\lvert \mu_{i}-\tfrac{rm}{n}\rvert.$$

Our final result is a generalization and refinement of the upper bound of \cite[Eq.~18]{LSC}.

\noindent\begin{theorem}\label{th:5-4}
Let $H$ be an $r$-uniform hypergraph with $n$ vertices and $m$ edges. Then
$$\LEE(H) \leq s+\tfrac{1}{2}M_{2}-1-\LE(H)-\tfrac{1}{2}\LE(H)^{2}+e^{\LE(H)}.$$
Equality is attained if and only if $H$ is an empty hypergraph.
\end{theorem}

\begin{proof}  Similar as before, we have that
\begin{align*}
\LEE(H)& = s+\tfrac{1}{2}M_{2}+\sum\limits_{k\geq 3}\tfrac{1}{k!}\sum_{i=1}^{s}(\mu_{i}-\tfrac{rm}{n})^{k}\leq s+\tfrac{1}{2}M_{2}+\sum\limits_{k\geq 3}\tfrac{1}{k!}\sum_{i=1}^{s}\lvert \mu_{i}-\tfrac{rm}{n}\rvert^{k}\\
&\leq s+\tfrac{1}{2}M_{2}+\sum\limits_{k\geq 3}\tfrac{1}{k!}\left(\sum_{i=1}^{s}\lvert \mu_{i}-\tfrac{rm}{n}\rvert\right)^{k}=s+\frac{1}{2}M_{2}+\sum\limits_{k\geq 3}\tfrac{1}{k!}\LE(H)^{k}\\&=s+\tfrac{1}{2}M_{2}-1-\LE(H)-\tfrac{1}{2}\LE(H)^{2}+e^{\LE(H)}.
\end{align*}

It is clear that equality holds if and only if $\mu_{1}=\mu_{2}=\cdots=\mu_{n}=\tfrac{rm}{n}=0$, i.e., $H$ is an empty hypergraph.
\end{proof}

Notice that $\LE(H)^{2}=\left(\sum\limits_{i=1}^{s}\lvert\mu_{i}-\tfrac{rm}{n}\rvert\right)^{2}\geq \sum\limits_{i=1}^{s}(\mu_{i}-\tfrac{rm}{n})^{2}=M_{2}$.
The bound is thus better that the more elementary bound

$$\LEE(H) \leq s-1-\LE(H)+e^{\LE(H)}.$$

For $r=2,$  this improves the upper bound of \cite[Eq.~18]{LSC}.

\end{document}